\documentclass{article}
\usepackage{graphicx,xcolor,textpos}
\usepackage{helvet}
\usepackage{amssymb,latexsym,amsmath,epsfig,amsthm,hyperref,tipa,comment}
\usepackage{enumitem}
\usepackage{pgfplots}
\newtheorem{main}{Theorem}

\newtheorem{vnincrease}{Proposition}
\newtheorem{cpminuspalmostincreases}[vnincrease]{Proposition}
\newtheorem{differentiable}[vnincrease]{Proposition}
\newtheorem{cpineqsecondpair}[vnincrease]{Proposition}

\newtheorem{boundsforphi}{Lemma}

\newtheorem{cdfmonotone}[boundsforphi]{Lemma}
\newtheorem{vnmonotone}[boundsforphi]{Lemma}
\newtheorem{boundsforall}[boundsforphi]{Lemma}
\newtheorem{expophi}[boundsforphi]{Lemma}
\newtheorem{expok}[boundsforphi]{Lemma}

\newtheorem{expo}[boundsforphi]{Lemma}
\newtheorem{arithgeo}[boundsforphi]{Lemma}

\newtheorem{kansopnstijgt}[boundsforphi]{Lemma}
\newtheorem{intervalhopping}[boundsforphi]{Lemma}
\newtheorem{recursived}[boundsforphi]{Lemma}

\newtheorem{nparbitrarilylarge}[boundsforphi]{Lemma}
\newtheorem{vnnotincrease}[boundsforphi]{Lemma}
\newtheorem{cpsmallerthanhalf}[boundsforphi]{Lemma}

\newtheorem{dipvsdiq}[boundsforphi]{Lemma}
\newtheorem{dipvsdiqtwo}[boundsforphi]{Lemma}

\newtheorem{cphi}{Corollary}
\newtheorem{cpsmallerthanpminusone}[cphi]{Corollary}
\newtheorem{continuouscp}[cphi]{Corollary}

\newtheorem{nplower}[cphi]{Corollary}

\begin{document}

\vspace*{-2cm}

\Large
 \begin{center}
On maximizing the number of heads when you need to set aside at least one coin every round \\ 

\hspace{10pt}

\large
Wouter van Doorn \\

\hspace{10pt}

\end{center}

\hspace{10pt}

\normalsize

\vspace{-10pt}

\centerline{\bf Abstract}
You play the following game: you start out with $n$ coins that all have probability $p$ to land heads. You toss all of them and you then need to set aside at least one of them, which will not be tossed again. Now you repeat the process with the remaining coins. This continues (for at most $n$ rounds) until all coins have been set aside. Your goal is to maximize the total number of heads you end up with. In this paper we will prove that there exists a constant $p_0 \approx 0.5495021777642$ such that, if $p_0 < p \le 1$ and the number of remaining coins is large enough, then it is optimal to set aside exactly one coin every round, unless all coins landed heads. When $\frac{1}{2} \le p \le p_0$, it is optimal to set aside exactly one coin every round, unless at most one coin came up tails. Let $v_{n,p}$ be the expected number of heads you obtain when using the optimal strategy. We will show that $v_{n,p}$ is larger than or equal to $v_{n-1,p} + 1$ for all $n \ge 3$ and all $p \ge \frac{1}{2}$. When $p < \frac{1}{2}$ there are infinitely many $n$ for which this inequality does not hold. Finally, we will see that for every $p$ with $0 < p \le 1$ the sequence $n - v_{n,p}$ is convergent.

\section{Introduction}
In the board game \textit{To Court The King}, players roll multiple dice and have to set aside at least one of them after every roll, before they re-roll the remaining ones. This board game inspired Joachim Breitner in $2015$ to wonder what the optimal strategy is and whether it is ever optimal te re-roll a $6$, assuming that the goal is to maximize the expected value of the sum of all dice. He then asked these questions online \mbox{\cite{jb}} and, over a year later, user joriki found a surprising result: if you start out with $200$ dice and you get exactly two $6$s, then the optimal decision is to set aside only one of them, and re-roll everything else! \\

This did answer one of the questions, but it remained unclear whether it was possible to easily phrase the optimal strategy in general. And this paper will unfortunately not shed a lot of light on this problem either. \\

What we will do however, is look at a slightly easier version of this game, where we work with coins instead of dice, and where the goal is to maximize the number of heads you end up with. With this version, we will show what the optimal strategy is, as long as the probability $p$ of heads is at least $\frac{1}{2}$. When $p$ is smaller than $\frac{1}{2}$ however (which is essentially the case when we play with dice) everything is a lot less clear. We will therefore mainly focus on the case $p \ge \frac{1}{2}$. \\

\section{A recursive formula}
Let $v_{n,p}$ be the expected number of heads you obtain with optimal play, if you start out with $n$ coins that all have probability $p$ of landing heads. Clearly, $v_{n,p}$ is an increasing function of $n$, and it is never worth it to set aside more coins that landed tails than necessary. Equally clearly, when $n = 0$ or $n = 1$, there is very little to do and $v_{0,p} = 0$ and $v_{1,p} = p$. \\

For $n = 2$ it only gets slightly more interesting; if we obtain two heads directly, we stop and get a payoff of $2$. If we obtain two tails, we are forced to set one of them aside and can expect a payoff of $p$. And if we obtain one heads and one tails, we should set aside the heads and reflip the tails for an expected pay-off of $1+p$. Since the probabilities of these three cases are $p^2, (1-p)^2$ and $2p(1-p)$ respectively, this gives $v_{2,p} = 2 \cdot p^2 + p \cdot (1-p)^2 + (1+p) \cdot 2p(1-p) = -p^3 + 3p$. \\

When $n = 3$, the first non-trivial decision needs to be made; if exactly two of the three tossed coins show heads, should you set aside one or both of them? \\

In general, $v_{n,p} \le n$ regardless of the valus of $n$ and $p$. This implies that if we obtain $n$ heads (which has probability $p^n$), we should set aside all of them for a pay-off of $n$. If we get no heads (which has probability $(1-p)^n$), then we should set aside exactly one of the coins that landed tails for an expected payoff of $v_{n-1,p}$. And if we obtain $j$ heads with $1 \le j \le n-1$ and we decide to set aside $i$ of them (with $1 \le i \le j$), then the expected payoff equals $v_{n-i,p} + i$. So we should choose $i$ in a way such that $v_{n-i,p} + i$ gets maximized. Since the probability to get $j$ heads is equal to $\binom{n}{j}p^j(1-p)^{n-j}$, this means that we obtain the following recursive formula:

\begin{equation}
v_{n,p} = np^n + v_{n-1,p} (1-p)^n + \sum_{j=1}^{n-1} \binom{n}{j}p^j(1-p)^{n-j} \left(\max_{1 \le i \le j} (v_{n-i,p} + i)\right) \label{recform}
\end{equation}

This formula will be the base of everything that follows.

\newpage
\section{Stating the optimal strategy for $p \ge \frac{1}{2}$}
Taking $p = \frac{1}{2}$ as an example, we get $v_{0, \frac{1}{2}} = 0, v_{1, \frac{1}{2}} = \frac{1}{2}, v_{2, \frac{1}{2}} = \frac{11}{8}$, by the calculation in the previous section. With equation (\ref{recform}), one can further calculate $v_{3,\frac{1}{2}} = \frac{19}{8}$ and $v_{4,\frac{1}{2}} = \frac{433}{128}$. These values are perhaps more surprising than they initially look. We get $v_{4,\frac{1}{2}} = v_{3,\frac{1}{2}} + \frac{129}{128} > v_{3,\frac{1}{2}} + 1$. So if you had to decide between tossing $4$ (fair) coins, or getting a free heads and tossing $3$ coins instead, you should choose the former! And this turns out to be a general rule; for all $p \ge \frac{1}{2}$ we get $v_{n,p} \ge v_{n-1,p} + 1$ for all $n \ge 3$, with strict inequality unless $n = 3$ and $p = \frac{1}{2}$.   \\

Moreover, note that the optimal strategy depends on whether $v_{2,p} > v_{1,p} + 1$ or $v_{2,p} < v_{1,p} + 1$ holds. If $n = 3$ and we roll two heads, then with $v_{2,p} < v_{1,p} + 1$ we should set aside both heads (for a pay-off of $v_{1,p} + 2$), whereas if $v_{2,p} > v_{1,p} + 1$ we should only set aside one of the coins showing heads (for a pay-off of $v_{2,p} + 1$). We claim that, in a sense, this is the only non-trivial decision to make when $p \ge \frac{1}{2}$; the main theorem we will prove is that whether you should settle for $n-1$ heads or not, is essentially the only variable in an optimal strategy. But before we prove our main theorem, let us state it first.

\begin{main} \label{Main}
Assume that you have $n \ge 3$ remaining coins and, after you toss them all, $j$ of those turn up heads. Further assume that the probability of heads is $p$, with $\frac{1}{2} \le p < 1$. It is then always optimal to follow one of these two strategies:

\begin{enumerate}[label=\Alph*]
	\item Set aside exactly one coin, unless $j = n$, in which case you take all of them.
	\item Set aside exactly one coin, unless $j = n$ or $j = n-1$. In both of these cases you take all heads.
\end{enumerate}

Moreover, there are two constants $\phi = \frac{\sqrt{5} - 1}{2} \approx 0.618$ and $p_0 \approx 0.5495021777642$, such that there are only three possibilities that can occur for a given $p \in [\frac{1}{2}, 1)$:

\begin{enumerate}
	\item For $\phi \le p < 1$ it is optimal to follow strategy A for all $n \ge 3$. 
	\item For $p_0 < p < \phi$ there exists an absolute constant $s'(p_0) > 0$ and a positive integer $n(p) = \left \lfloor \frac{\log(p - p_0) + \log\left(s'(p_0)\right)}{\log(p_0)} + o(1) \right \rfloor$\footnote{Here, $o(1)$ is a quantity that converges to $0$ as $n(p)$ goes to infinity, which happens when $p$ converges to $p_0$ from above.} such that it is optimal to follow strategy A when $n > n(p)$, but one should switch to strategy B when the remaining number of coins becomes smaller than or equal to $n(p)$.
	\item For $\frac{1}{2} \le p \le p_0$ it is optimal to follow strategy B for all $n \ge 3$.
\end{enumerate}

\end{main}

\newpage
\section{Splitting up our Main Theorem into four parts} \label{splitting}
We claim that the entirety of Theorem \ref{Main} follows from equation (\ref{recform}) and four propositions. In order to be able to state these propositions, define $s_n(p) := v_{n,p} - n + 1 - p$ and $s(p) := \displaystyle \lim_{n \rightarrow \infty} s_n(p)$. A priori this limit might not exist of course, but we will see that it does. Finally, define $n(p)$ to be the smallest integer $n \ge 2$ for which $s_n(p) \ge 0$, if it exists.

\begin{vnincrease} \label{vnincrease}
For all $p$ with $\frac{1}{2} \le p < 1$ and all $n \ge 3$ we have $v_{n,p} \ge v_{n-1,p} + 1$, with strict inequality unless $p = \frac{1}{2}$ and $n = 3$. If $p$ is larger than or equal to $\phi$, then the inequality $v_{n,p} \ge v_{n-1,p} + 1$ already holds for $n = 2$, with strict inequality unless $p = \phi$. In particular, if $p \ge \frac{1}{2}$, then $s_n(p)$ is for $n \ge 2$ an increasing function of $n$. 
\end{vnincrease}

\begin{cpminuspalmostincreases} \label{cpincrease}
The sequence $s_n(p)$ converges uniformly to $s(p)$ in the interval $[\frac{1}{2}, 1]$. Moreover, there exists a constant $p_0 \approx 0.5495021777642$ with $s(p_0) = 0$ such that $s(p) < 0$ for $\frac{1}{2} \le p < p_0$ and $s(p) > 0$ for $p_0 < p < 1$. In particular, $n(p)$ exists for all $p > p_0$.
\end{cpminuspalmostincreases}

\begin{cpineqsecondpair} \label{cpineqsecondpair}
For every $\epsilon > 0$ there exists a $\delta > 0$ such that for all $p \in (p_0, p_0 + \delta)$ we get $(1 - \epsilon) p^{n(p)+1} < s(p) < (1 + \epsilon) p^{n(p)}$.
\end{cpineqsecondpair}

\begin{differentiable} \label{differentiable}
The function $s(p)$ is differentiable in the point $p = p_0$, and $s'(p_0) > 0$.
\end{differentiable}

In future sections we will tackle these propositions one by one. But first we will show that, if we combine them, we are able to deduce Theorem \ref{Main}.

\begin{proof}[Proof of Theorem \ref{Main}]
To see why the above four propositions together indeed imply Theorem \ref{Main}, first note that iterating the inequality $v_{i,p} \ge v_{i-1,p} + 1 $ for $3 \le i \le n-1$ gives, after a change of variables, $v_{n-1,p} + 1 \ge v_{n-i,p} + i$ for all $i$ with $1 \le i \le n-2$. Equation (\ref{recform}) therefore simplifies as follows:

\begin{equation}
\begin{split}
v_{n,p} &= np^n + v_{n-1,p}(1-p)^n + \sum_{j=1}^{n-1} \binom{n}{j}p^j(1-p)^{n-j} \left(\max_{1 \le i \le j} (v_{n-i,p} + i)\right)  \\
\omit \span = np^n + v_{n-1,p}(1-p)^n + np^{n-1}(1-p) \max (v_{n-1,p} + 1, v_{1,p} + n - 1) \\
&+ \sum_{j=1}^{n-2} \binom{n}{j}p^j(1-p)^{n-j} (v_{n-1,p} + 1) \\
\omit \span = np^n + v_{n-1,p}(1-p)^n + np^{n-1}(1-p) \max (v_{n-1,p} + 1, v_{1,p} + n - 1) \\
&+ (v_{n-1,p} + 1) \left(1 - (1-p)^n - np^{n-1}(1-p) - p^n\right) \\
\omit \span = v_{n-1,p} + 1 + p^n(n - 1 - v_{n-1,p}) + np^{n-1}(1-p) \max(0, n - 2 + p - v_{n-1,p}) - (1-p)^n \label{recgeneral}
\end{split}
\end{equation}

The only unknown in this formula is the value of $\max(0, n - 2 + p - v_{n-1,p})$. In other words, the only difference in optimal strategies for different values of $p$, comes from whether $v_{n-1,p}$ is larger than $n - 2 + p$ or not. This decides whether you should, in the case of $n-1$ heads, take one or all of them. \\

By Proposition \ref{vnincrease} we know that $s_n(p) = v_{n,p} - n + 1 - p$ is for $n \ge 2$ a monotonically increasing function of $n$. Since $v_{n,p} \le n$, we furthermore know that $s_n(p)$ is upper bounded by $1 - p$. As $s(p)$ is defined as a limit of an increasing and bounded sequence, it exists. The inequality $v_{n,p} \ge v_{n-1,p} + 1$ furthermore shows that, as soon as $v_{n-1,p} \ge n - 2 + p$ for some $n \ge 3$, we get $v_{k,p} > k - 1 + p$ for all $k$ larger than $n-1$ as well. We thereby obtain the three distinct possibilities mentioned in Theorem \ref{Main}; either $v_{n,p} \ge n - 1 + p$ for all $n \ge 2$ (and we get possibility $1$), or $v_{n,p} > n - 1 + p$ for large enough $n$ (possibility $2$), or $v_{n,p} < n - 1 + p$ for all $n \ge 2$ (possibility $3$). The first and second possibility arise when $s(p) > 0$, and the third possibility arises when $s(p) \le 0$. This also explains the reason we defined $s_n(p)$ and $s(p)$ this way. \\

The only thing left to do is to show $n(p) = \left \lfloor \frac{\log(p - p_0) + \log\left(s'(p_0)\right)}{\log(p_0)} + o(1) \right \rfloor$ for $p_0 < p < \phi$, and we claim that this estimate follows from Propositions \ref{cpineqsecondpair} and \ref{differentiable}. To prove this, let $\epsilon > 0$ be arbitrary and define $\epsilon_1 = \min\left(\frac{s'(p_0)}{2}, \frac{\epsilon s'(p_0)}{16}, \frac{\epsilon}{8}, \frac{1}{2}\right)$. \\

Now, first of all, let $\delta_1 > 0$ be such that for all $p \in (p_0, p_0 + \delta_1)$ we have the inequalities $(1 - \epsilon_1) p^{n(p)+1} < s(p) < (1 + \epsilon_1) p^{n(p)}$ from Proposition \ref{cpineqsecondpair}. \\

Secondly, since $s(p)$ is differentiable in $p_0$ by Proposition \ref{differentiable} and $s(p_0)$ is equal to $0$, we obtain by the definition of differentiability that there exists a $\delta_2 > 0$ such that for all $p \in (p_0, p_0 + \delta_2)$ we have the inequalities $(s'(p_0) - \epsilon_1)(p - p_0) < s(p) < (s'(p_0) + \epsilon_1)(p - p_0)$. \\

Finally, define $\delta = \min\left(\delta_1, \delta_2, \frac{1}{20}, \left(\frac{\epsilon}{32}\right)^2, \frac{\epsilon}{32\left|\log\left(s'(p_0)\right)\right| + 1}\right)$. \\

For all $p \in (p_0, p_0 + \delta) \subseteq (0.5, 0.6)$ we then have, by the above, the inequalities $(s'(p_0) - \epsilon_1)(p - p_0) < s(p) < (1 + \epsilon_1) p^{n(p)}$. By dividing by $(1 + \epsilon_1) $ and taking the logarithm on both sides, we get $n(p) < \frac{\log(p - p_0)}{\log(p)} + \frac{\log(s'(p_0) - \epsilon_1)}{\log(p)} - \frac{\log(1 + \epsilon_1)}{\log(p)}$. \\

In order to slightly rewrite this upper bound we are going to apply $\frac{-1}{\log(p_0)} < \frac{-1}{\log(p)} < 2$ and make use of the two well-known bounds (both valid for all $x, y > 0$) $\log(x + y) < \log(x) + \frac{y}{x}$ and $\log(x) < \sqrt{x}$.

\begin{align*}
n(p) &< \frac{\log(p - p_0)}{\log(p)} + \frac{\log(s'(p_0) - \epsilon_1)}{\log(p)} - \frac{\log(1 + \epsilon_1)}{\log(p)} \\
&< \frac{\log(p - p_0)}{\log(p)} + \frac{\log\big(s'(p_0)\big)}{\log(p)} + \frac{2\epsilon_1}{s'(p_0) - \epsilon_1} + 2\epsilon_1 \\
&= \frac{\log(p - p_0)}{\log(p_0)} + \frac{\big(\log(p_0) - \log(p)\big)\log(p - p_0)}{\log(p)\log(p_0)} + \frac{\log\big(s'(p_0)\big)}{\log(p_0)} \\
&\hspace{10pt}+ \frac{\big(\log(p_0) - \log(p)\big)\log\big(s'(p_0)\big)}{\log(p)\log(p_0)} + \frac{2\epsilon_1}{s'(p_0) - \epsilon_1} + 2\epsilon_1 \\
&< \frac{\log(p - p_0) + \log\big(s'(p_0)\big)}{\log(p_0)} + 8(p_0 - p)\log(p - p_0) \\
&\hspace{10pt}+ 8(p - p_0)\left|\log\big(s'(p_0)\big)\right| + \frac{4\epsilon_1}{s'(p_0)} + 2\epsilon_1 \\
&< \frac{\log(p - p_0) + \log\big(s'(p_0)\big)}{\log(p_0)} + 8\sqrt{\delta} + 8\delta\left|\log\big(s'(p_0)\big)\right| + \frac{1}{4}\epsilon + \frac{1}{4}\epsilon \\
&\le \frac{\log(p - p_0) + \log\big(s'(p_0)\big)}{\log(p_0)} + \epsilon
\end{align*}

On the other hand, in order to obtain a lower bound on $n(p)$ for all $p \in (p_0, p_0 + \delta) \subseteq (0.5, 0.6)$, we combine the inequalities $(1 - \epsilon_1) p^{n(p)+1} < s(p)$ and $s(p) < (s'(p_0) + \epsilon_1)(p - p_0)$, divide by $(1 - \epsilon_1)$, take the logarithm, and then make use of the bound $\log(1 - x) > -2x$ which holds for all $x < 0.79$.

\begin{align*}
n(p) &> \frac{\log(p - p_0)}{\log(p)} + \frac{\log(s'(p_0) + \epsilon_1)}{\log(p)} - \frac{\log(1 - \epsilon_1)}{\log(p)} - 1 \\
&> \frac{\log(p - p_0) + \log\big(s'(p_0)\big)}{\log(p_0)} - 1 - \frac{2\epsilon_1}{s'(p_0)} - 4\epsilon_1 \\
&> \frac{\log(p - p_0) + \log\big(s'(p_0)\big)}{\log(p_0)} - 1 - \epsilon
\end{align*}

Since $n(p) \in \mathbb{N}$ and there is for small enough $\epsilon$ often only a single unique integer in the interval $\left(\frac{\log(p - p_0) + \log(s'(p_0))}{\log(p_0)} - 1 - \epsilon, \frac{\log(p - p_0) + \log(s'(p_0))}{\log(p_0)} + \epsilon\right)$, we get the estimate $n(p) = \left \lfloor \frac{\log(p - p_0) + \log\left(s'(p_0)\right)}{\log(p_0)} + o(1) \right \rfloor$.
\end{proof}

All in all, Theorem \ref{Main} essentially boils down to the inequalities $v_{n,p} \ge v_{n-1,p} + 1$, $s(p) > 0$ and $(1 - \epsilon) p^{n(p)+1} < s(p) < (1 + \epsilon) p^{n(p)}$. Here, the first inequality holds for all $p \ge \frac{1}{2}$ and $n \ge 3$, the second inequality holds if, and only if, $p > p_0$, and the final inequalities hold for all $p > p_0$ sufficiently close to $p_0$. \\

To prove the first inequality, we will split up the interval $[\frac{1}{2}, 1)$ into various sub-intervals, starting with the case $p \ge \phi$. To prove the second inequality, we will see that $s(p)$ is a continuous function that is negative for some values of $p$ and positive for other values of $p$ and thereby deducing the existence of a $p_0$ such that $s(p_0)$ is exactly equal to $0$. We then show that this value is unique, before moving on to the aforementioned bounds on $s(p)$ and its differentiability.

\section{The inequality $v_{n,p} \ge v_{n-1,p} + 1$ for $\phi \le p < 1$}
In this section we will prove Proposition \ref{vnincrease} for the sub-interval $[\phi, 1)$.

\begin{boundsforphi} \label{boundsforphi}
Let $p$ be such that $\phi \le p < 1$. Then $v_{n-1,p} + 1 \le v_{n,p} < n - \frac{(1-p)^{n+1}}{p^{n+1}}$ for all $n \ge 2$, with strict inequality for the lower bound, unless $n = 2$ and $p = \phi$.
\end{boundsforphi}

\begin{proof}
We will prove these inequalities via induction. First of all, as we noticed before, $v_{2,p} = -p^3 + 3p$. And proving that this is larger than or equal to $v_{1,p} + 1 = p + 1$ is equivalent with proving that their difference $v_{2,p} - v_{1,p} - 1 = -p^3 + 2p - 1 = (1-p)(p^2 + p - 1)$ is non-negative for $p \in [\phi, 1)$. Since $1-p > 0$ and $p^2 + p \ge 1$ for $p \ge \phi$, this is immediate. This establishes the base case for the lower bound. \\

As for a base case for the upper bound, we need to check that the inequality $v_{2,p} = -p^3 + 3p < 2 - \frac{(1-p)^{3}}{p^{3}}$ holds for $\phi \le p < 1$. Multiplying this out leaves us with $-p^6 + 3p^4 - 3p^3 + 3p^2 - 3p + 1 = -(p-1)^2(p^4 + 2p^3 + p - 1) < 0$. Since $-(p-1)^2 < 0$, it is now sufficient to show $p^4 + 2p^3 + p > 1$. This is not hard to see however, since $p^4 + 2p^3 + p > 2p^3 + p > p^2 + p \ge 1$ for $p \ge \phi$. We therefore have a base for the upper bound as well. \\

Assume now $n \ge 3$ and that for all $k$ with $2 \le k \le n-1$ we have $v_{k-1,p} + 1 \le v_{k,p} < k - \frac{(1-p)^{k+1}}{p^{k+1}}$. We will then prove that these inequalities (but with strict inequality for the lower bound) hold for $k = n$ as well. \\

Since $v_{2,p} \ge p+1$ for $p \ge \phi$, we see by induction that equation (\ref{recgeneral}) in this case further simplifies to the following:

\begin{equation}
v_{n,p} = v_{n-1,p} + 1 + p^n(n - 1 - v_{n-1,p}) - (1-p)^n \label{recformtwo}
\end{equation}

Now, on the one hand, applying $v_{n-1,p} < n - 1 - \frac{(1-p)^{n}}{p^{n}}$ to this entire equation (which we are allowed to do) gives us a proof for the upper bound.

\begin{align*}
v_{n,p} &= v_{n-1,p} + 1 + p^n(n - 1 - v_{n-1,p}) - (1-p)^n \\
&< n - \frac{(1-p)^{n}}{p^{n}} + p^n\left(n-1 - \left(n - 1 - \frac{(1-p)^{n}}{p^{n}}\right)\right) - (1-p)^n \\
&= n - \frac{(1-p)^{n}}{p^{n}} \\
&< n - \frac{(1-p)^{n+1}}{p^{n+1}}
\end{align*}

While, on the other hand, only applying $v_{n-1,p} < n - 1 - \frac{(1-p)^{n}}{p^{n}}$ to the second occurrence of $v_{n-1,p}$ in equation (\ref{recformtwo}) gives us a proof for the lower bound.

\begin{align*}
v_{n,p} &= v_{n-1,p} + 1 + p^n(n - 1 - v_{n-1,p}) - (1-p)^n \\
&> v_{n-1,p} + 1 + p^n\left(n-1 - \left(n - 1 - \frac{(1-p)^{n}}{p^{n}}\right)\right) - (1-p)^n \\
&= v_{n-1,p} + 1 \qedhere
\end{align*}
\end{proof}

As we have mentioned before, the inequality $v_{n,p} \ge v_{n-1,p} +1$ (which we have now proven for $p \ge \phi$) implies the existence of a constant $s(p) \le 1 - p$ such that $s_n(p) = v_{n,p} - n + 1 - p$ converges to $s(p)$. Since $s(p) \le 1 - p$, it is natural to ask whether equality ever occurs for $p < 1$. We will show that this is not the case; $s(p)$ is strictly smaller than $1 - p$ for all $p$ with $\phi \le p < 1$, and we will later extend this inequality to all $p < 1$.

\begin{expophi} \label{expophi}
For all $p$ with $\phi \le p < 1$ we have $s(p) < 1 - p - (1-p)^{\frac{\log(1-p)}{\log(p)} + 2}$.
\end{expophi}

\begin{proof}
Choose $N_0 = \left \lceil \frac{\log(1-p)}{\log(p)} \right \rceil$ so that $p^{N_0} \le 1 - p$ and define $\delta = N_0 - v_{N_0,p} \ge (1-p)^{N_0}$. Since $v_{k,p} \ge v_{k-1,p} + 1$ for all $k \ge N_0$, note that $k-1-v_{k-1,p} \le \delta$ for all $k > N_0$. Now we can bound the value of $s(p)$ by applying equation (\ref{recformtwo}).

\begin{align*}
s(p) &= 1 - p + \lim_{N \rightarrow \infty} \big[v_{N,p} - N\big] \\
&= 1 - p + \lim_{N \rightarrow \infty} \big[v_{N,p} - v_{N_0,p} - N + v_{N_0,p}\big] \\
&= 1 - p - \delta + \lim_{N \rightarrow \infty} \big[(v_{N,p} - v_{N_0,p}) - (N - N_0)\big] \\
&= 1 - p - \delta + \lim_{N \rightarrow \infty} \left[\sum_{k = N_0+1}^N v_k - v_{k-1} - 1\right] \\
&= 1 - p - \delta + \sum_{k = N_0+1}^{\infty} (k-1-v_{k-1})p^k - (1-p)^k \\
&< 1 - p - \delta + \sum_{k = N_0+1}^{\infty} \delta p^k \\
&= 1 - p  - \delta \left(1 - \frac{p^{N_0 + 1}}{1-p} \right) \\
&\le 1 - p - \delta (1 - p) \\
&\le 1 - p - (1-p)^{N_0+1} \\
&\le 1 - p - (1-p)^{\frac{\log(1-p)}{\log(p)} + 2} \qedhere
\end{align*}
\end{proof}

\section{The inequality $v_{n,p} \ge v_{n-1,p} + 1$ for $\frac{1}{2} \le p \le \phi$}
The goal of this section is to extend the lower bound from Lemma \ref{boundsforphi} to all $p \ge \frac{1}{2}$.

\begin{boundsforall} \label{boundsforall}
Let $p$ be such that $\frac{1}{2} \le p \le \phi$. Then $v_{n,p} \ge v_{n-1,p} + 1$ for all $n \ge 3$, with strict inequality, unless $n = 3$ and $p = \frac{1}{2}$.
\end{boundsforall} 

\begin{proof}
To start off, let us deal with $n = 3$ and $n = 4$ separately. If $p \le \phi$, then $v_{2,p} = -p^3 + 3p \le p + 1 = v_{1,p} + 1$. Equation (\ref{recform}) then implies $v_{3,p} = -2p^6 + 3p^5 + 3p^4 - 7p^3 + 6p$, so that $v_{3,p} - v_{2,p} - 1 = -2p^6 + 3p^5 + 3p^4 - 6p^3 + 3p - 1 = (1-p)^3(p+1)^2(2p-1)$. This is immediately seen to be positive for $\frac{1}{2} < p \le \phi$ and equal to zero for $p = \frac{1}{2}$. \\

As for $n = 4$, since $\displaystyle \max_{1 \le i \le j} (v_{n-i,p} + i) = v_{3,p} + 1$ for $j \le 2$ and $\displaystyle \max_{1 \le i \le j} (v_{n-i,p} + i) \ge v_{1,p} + 3$ for $j = 3$, one can check that equation (\ref{recform}) implies the following:

\begin{equation*}
v_{4,p} \ge -6p^{10} + 17p^9 - 3p^8 - 33p^7 + 26p^6 + 17p^5 - 23p^4 + 5p^3 - 6p^2 + 10p
\end{equation*}

And we want to show that this is larger than $v_{3,p} + 1$. This is seen as follows:

\begin{align*}
v_{4,p} - v_{3,p} - 1 &\ge -6p^{10} + 17p^9 - 3p^8 - 33p^7 + 28p^6 + 14p^5 - 26p^4 + 12p^3 - 6p^2 + 4p - 1 \\
&= (1-p)^2\left(6p^7\left(\frac{5}{6} - p\right) + 13p^6 - 12p^5 - 9p^4 + 8p^3 - p^2 + 2p - 1\right) \\
&> (1-p)^2(13p^6 - 12p^5 - 9p^4 + 8p^3 - p^2 + 2p - 1) \\
&= (1-p)^3(-13p^5 - p^4 + 8p^3 + p - 1) \\
&> (1-p)^3(-16p^5 + 8p^3 + p - 1) \\
&= (1-p)^3(2p-1)\big(1 - (2p^2 + p)(4p^2 - 1)\big) \\
&\ge (1-p)^3(2p-1)(0.27) \\
&\ge 0
\end{align*}

Where in the third and second to last line we used $p \le \phi$. We will therefore assume $n \ge 5$ for the rest of this section. \\

The proof of Lemma \ref{boundsforall} now rests on the following lemma that we can iteratively apply:

\begin{intervalhopping} \label{intervalhopping}
Let $q_1, q_2$ be such that $\frac{1}{2} \le q_2 < q_1 \le \phi$ and assume that Lemma \ref{boundsforall} holds for $p = q_1$. Then Lemma \ref{boundsforall} also holds for all $p \in [q_2, q_1)$, as long as the following inequality is met:

\begin{equation}
q_2^5\big(1 - q_1 - s(q_1)\big) + 5q_2^{4}(1-q_2)\max\big(0, q_2 - q_1 - s(q_1)\big) > (1 - q_2)^{5} \label{hopper}
\end{equation}
\end{intervalhopping}

Lemma \ref{boundsforall} then follows from Lemma \ref{intervalhopping} if there exists a sequence $\phi = r_1 > r_2 > \ldots > r_m = \frac{1}{2}$ such that with $q_1 = r_i$ and $q_2 = r_{i+1}$ equation (\ref{hopper}) holds for all $i$ with $1 \le i \le m-1$. \\

To prove Lemma \ref{intervalhopping}, we need to check that a few functions are increasing.

\begin{cdfmonotone} \label{cdfmonotone}
Let $p, q \in [0, 1]$, let $j$ and $n$ be positive integers with $j \le n$, and define the partial sum $B_j(p) := \displaystyle \sum_{k=j}^n \binom{n}{k} p^k (1-p)^{n-k}$. If $p > q$, then $B_j(p) > B_j(q)$.
\end{cdfmonotone}

\begin{proof}
We will show that the derivative of $B_j(p)$ with respect to $p$ is positive, which is sufficient. Define $S_j = \displaystyle \sum_{k=j}^n \binom{n}{k} kp^k (1-p)^{n-k}$ and first of all note that $S_j \ge jB_j(p)$. Secondly, by the expected value of the binomial distribution, $np - S_j = \displaystyle \sum_{k=0}^{j-1} \binom{n}{k} kp^k (1-p)^{n-k} < j(1 - B_j(p))$. Now we calculate the derivative of $B_j(p)$.

\begin{align*}
B_j'(p) &= \sum_{k=j}^n \binom{n}{k} kp^{k-1} (1-p)^{n-k} + \sum_{k=j}^n \binom{n}{k} (k-n)p^k(1-p)^{n-k-1} \\
&= \frac{S_j}{p} + \frac{S_j}{1-p} - \frac{nB_j(p)}{1-p} \\
&= \frac{1}{p(1-p)}\big(S_j - npB_j(p) \big) \\
&= \frac{1}{p(1-p)}\big(S_j - (np - S_j)B_j(p) + S_jB_j(p) \big) \\
&> \frac{1}{p(1-p)}\big(S_j - j\big(1 - B_j(p)\big)B_j(p) + S_jB_j(p) \big) \\
&\ge \frac{1}{p(1-p)}\big(S_j - S_j\big(1 - B_j(p)\big) + S_jB_j(p) \big) \\
&= 0 \qedhere
\end{align*}
\end{proof}

\begin{vnmonotone} \label{vnmonotone}
Let $p, q$ be such that $0 \le q < p \le 1$. Then $v_{n,p} > v_{n,q}$ for all $n \in \mathbb{N}$.
\end{vnmonotone}

\begin{proof}
For $n = 1$ the inequality is rather trivial, so assume it is true by induction for all $k < n$. In a game with $n$ coins and probability of heads equal to $q$, let $q_j$ be the probability that you decide to set aside $j$ heads and note that $q_j \le B_j(q)$. Assuming you play optimally, this implies $v_{n,q} = q_0v_{n-1,q} + \displaystyle \sum_{j = 1}^n q_j(v_{n-j,q}+j)$. Now, from Lemma \ref{cdfmonotone} we know $B_j(p) > B_j(q) \ge q_j$ for all $j \in \{1, 2, \ldots, n\}$. This implies that, in the game with $n$ coins and probability $p$ of heads, one can still apply the strategy to set aside $j$ heads with probability $q_j$. We then get $v_{n,p} \ge q_0v_{n-1,p} + \displaystyle \sum_{j = 1}^n q_j(v_{n-j,p}+j) > q_0v_{n-1,q} + \displaystyle \sum_{j = 1}^n q_j(v_{n-j,q}+j) = v_{n,q}$.
\end{proof}

\begin{kansopnstijgt} \label{kansopnstijgt}
Let $p,q$ be such that $\frac{1}{2} \le q < p \le \phi$. Then for all $n \ge 2$ we have the inequality $p^n(1-p) > q^n(1-q)$.
\end{kansopnstijgt}

\begin{proof}
We will prove that, for all fixed $n \ge 2$, the function $p^n(1-p)$ is increasing as a function of $p$. Its derivative is equal to $np^{n-1}(1-p) - p^n$. Dividing this by $p^{n-1}$ gives us $n(1-p) - p \ge 2(1-p) - p = 2 - 3p > 0$, since $p \le \phi < \frac{2}{3}$.
\end{proof}

\begin{proof}[Proof of Lemma \ref{intervalhopping}]
Let $q_1, q_2$ be as in the statement of Lemma \ref{intervalhopping}, assume $p \in [q_2, q_1) \subseteq [\frac{1}{2}, \phi)$ and let $n \ge 5$ be a positive integer. We then apply equation (\ref{recform}) (in the same way as in the derivation of equation (\ref{recgeneral})) and the previous lemmas to get the following:

\begin{align*}
v_{n,p} &\ge v_{n-1,p} + 1 + p^n (n - 1 - v_{n-1,p}) + np^{n-1}(1 - p) \max(0, n - 2 + p - v_{n-1,p}) - (1 - p)^n \\
&\ge v_{n-1,p} + 1 + q_2^n (n - 1 - v_{n-1, q_1}) + nq_2^{n-1}(1 - q_2) \max(0, n - 2 + q_2 - v_{n-1, q_1}) - (1 - q_2)^n \\
&> v_{n-1,p} + 1 + q_2^n (1 - q_1 - s(q_1)) + nq_2^{n-1}(1 - q_2)\max\big(0, q_2 - q_1 - s(q_1)\big) - (1 - q_2)^n \\
&\ge v_{n-1,p} + 1 + (1 - q_2)^{n-{5}} \bigg(q_2^5\big(1 - q_1 - s(q_1)\big) + 5q_2^{4}(1-q_2)\max\big(0, q_2 - q_1 - s(q_1)\big) - (1 - q_2)^{5} \bigg) \\
&> v_{n-1,p} + 1  \qedhere
\end{align*}

\end{proof}

To actually make use of Lemma \ref{intervalhopping}, we need an upper bound on $s(q_1)$ to be able to check equation (\ref{hopper}). We will therefore look at the speed of convergence of $s_{n}(p)$ to $s(p)$. And to do this, we need a small preliminary lemma, the proof of which is left to the reader.

\begin{arithgeo}
For all $p$ with $0 \le p < 1$ and all $n \in \mathbb{N}$ we have the following equality:

\begin{equation*}
\sum_{k = n+1}^{\infty} kp^k = p^{n+1} \left(\frac{n}{1-p} + \frac{1}{(1-p)^2}\right)
\end{equation*}
\end{arithgeo}

\begin{expo} \label{expo}
For all $p \in [\frac{1}{2}, 1)$ for which $v_{n,p} \ge v_{n-1,p} + 1$ holds for all $n \ge 3$, we have the inequality $s(p) - s_{n}(p) < p^{n+1} \left(\frac{n}{1-p} + \frac{1}{(1-p)^2}\right)$ for all $n \in \mathbb{N}$. 
\end{expo}

\begin{proof}
With $s(p) \le 1-p$ the inequality is quickly checked for $n = 1$, so assume $n \ge 2$. The proof then starts off very similar to the proof of Lemma \ref{expophi}. The main difference is that Lemma \ref{expophi} applied equation (\ref{recformtwo}) which used the fact that, in that case, $v_{n,p} \ge v_{n-1,p} + 1$ already held for $n \ge 2$. In this case we only assume that it holds for $n \ge 3$, which means that we need to make use of equation (\ref{recgeneral}) here.

\begin{align*}
s(p) - s_n(p) &= \lim_{N \rightarrow \infty} \big[(v_{N,p} - N + 1 - p) - (v_{n,p} - n + 1 - p)\big] \\
&= \lim_{N \rightarrow \infty} \big[(v_{N,p} - v_{n,p}) - (N - n)\big] \\ 
&= \lim_{N \rightarrow \infty} \left[\sum_{k = n+1}^N v_{k,p} - v_{k-1,p} - 1\right] \\ 
&= \sum_{k = n+1}^{\infty} \Big(p^k(k - 1 - v_{k-1,p}) + kp^{k-1}(1-p) \max(0, k - 2 + p - v_{k-1,p}) - (1-p)^k\Big) \\
&< \sum_{k = n+1}^{\infty} \left(p^k + kp^{k-1}(1-p)\right)(k - 1 - v_{k-1,p}) \\
&\le \sum_{k = n+1}^{\infty} \left(p^k + kp^{k-1}(1-p)\right)(2 - v_{2,\frac{1}{2}}) \\
&= \sum_{k = n+1}^{\infty} \frac{5}{8}\left(p^k + kp^{k-1}(1-p)\right) \\
&< \sum_{k = n+1}^{\infty} kp^k  \\
&= p^{n+1} \left(\frac{n}{1-p} + \frac{1}{(1-p)^2}\right) \qedhere
\end{align*}
\end{proof}

\begin{cphi} \label{cphi}
For $p = \phi$ we have $s(p) < 0.121$.
\end{cphi}

\begin{proof}
Since $v_{n,p} \ge v_{n-1,p} + 1$ is known to hold for $p = \phi$ by Lemma \ref{boundsforphi}, we can (preferably with a computer) calculate $v_{30,\phi} \approx 29.7389$ and then apply Lemma \ref{expo} to obtain $s(p) < 0.121$.
\end{proof}

Choosing $q_1 = \phi$ and $q_2 = 0.57$, one can check that equation (\ref{hopper}) is satisfied by applying Corollary \ref{cphi}, so that Lemma \ref{boundsforall} actually holds for all $p \ge 0.57$. And we can now iterate this; choose $q_1 = 0.57$, calculate $v_{30,q_1} \approx 29.6038$ to obtain $s(q_1) < 0.034$ with Lemma \ref{expo}, and now equation (\ref{hopper}) can be checked with $q_2 = 0.55$. \\

By continuing in this manner with the following sequence, the entire interval $[\frac{1}{2}, \phi]$ is covered: $\phi, 0.57, 0.55, 0.54, 0.535, 0.529, 0.523, 0.518, 0.513, 0.508, 0.504, \frac{1}{2}$. \\

For completeness' sake, here are tables with the aforementioned sequence and upper bounds on the corresponding values of $v_{q_1,30}$ and $s(q_1) < s_{30}(q_1) + 0.0001$:
\newpage
\begin{table}[ht]
\begin{tabular}{|l|llllll|} \hline 
$q_1$    & $\phi$   & $0.57$  & $0.55$ & \hspace{4pt} $0.54$  & \hspace{4pt} $0.535$  & \hspace{4pt} $0.529$  \\ \hline 
$v_{q_1,30}$ & $29.7389$   & $29.6038$  & $29.5508$ & \hspace{4pt} $29.5253$  & \hspace{4pt} $29.5117$  & \hspace{4pt} $29.4947$  \\ \hline 
$s(q_1)$ & $0.121$  & $0.034$ & $0.001$ & $-0.014$ & $-0.023$ & $-0.034$ \\ \hline 
$q_2$    & $0.57$   & $0.55$  & $0.54$  & \hspace{4pt} $0.535$  & \hspace{4pt} $0.529$ & \hspace{4pt} $0.523$ \\ \hline 
\end{tabular}
\end{table}

\begin{table}[ht]
\begin{tabular}{|l|lllll|} \hline 
$q_1$    & \hspace{4pt} $0.523$  & \hspace{4pt} $0.518$  & \hspace{4pt} $0.513$  & \hspace{4pt} $0.508$  & \hspace{4pt} $0.504$ \\ \hline 
$v_{q_1,30}$ & $29.4769$   & $29.4614$  & $29.4452$ & \hspace{4pt} $29.4285$  & \hspace{4pt} $29.4146$ \\ \hline 
$s(q_1)$ & $-0.046$ & $-0.056$ & $-0.067$ & $-0.079$ & $-0.089$ \\ \hline
$q_2$    & \hspace{4pt} $0.518$  & \hspace{4pt} $0.513$  & \hspace{4pt} $0.508$  & \hspace{4pt} $0.504$ & \hspace{4pt} $0.5$ \\ \hline 
\end{tabular}
\end{table}
\end{proof}

Since $v_{n,p} \ge v_{n-1,p} + 1$ is now proven for all $p \ge \frac{1}{2}$, we conclude that $s(p)$ exists for all $p \ge \frac{1}{2}$ and it actually leads to two more colloraries. 

\begin{cpsmallerthanpminusone}
For all $p$ with $\frac{1}{2} \le p < 1$ we have $s(p) < 1 - p$.
\end{cpsmallerthanpminusone}

\begin{proof}
For $p \ge \phi$ this is Lemma \ref{expophi}. For $p < \phi$ we get $s(p) \le s(\phi) + \phi - p < 1 - p$.
\end{proof}

\begin{continuouscp} \label{cont}
In the interval $[\frac{1}{2}, 1]$, $s_n(p)$ converges uniformly to $s(p)$. In particular, the function $s(p)$ is continuous in this interval.
\end{continuouscp}

\begin{proof}
The second sentence follows from the first as we claim that $s_n(p)$ is continuous, which makes $s(p)$ a uniform limit of continuous functions, which is well-known to imply continuity. Indeed, recall that $s_n(p) = v_{n,p} - n + 1 - p$ and note that $v_{n,p}$ is continuous, since it is by equation (\ref{recform}) a sum of maxima of continuous functions of $p$. \\

With that out of the way, let us show uniform convergence of $s_n(p)$ to $s(p)$. Let $\epsilon > 0$ be given, assume $\epsilon < 1 - \phi$ without loss of generality, and choose $N$ large enough such that $p^{n+1} \left(\frac{n}{1-p} + \frac{1}{(1-p)^2}\right) < \epsilon$ for $p = 1 - \epsilon$ and all $n \ge N$. Note that this implies $p^{n+1} \left(\frac{n}{1-p} + \frac{1}{(1-p)^2}\right) < \epsilon$ for all $p < 1 - \epsilon$ and all $n \ge N$ as well. Therefore, $|s(p) - s_n(p)| < \epsilon$ for all $p \in [\frac{1}{2}, 1 - \epsilon]$ by Lemma \ref{expo}, whereas for all $p \in (1 - \epsilon, 1]$ we have $|s(p) - s_n(p)| \le 1 - p - s_n(p) \le 1 - p - s_1(p) = 1 - p < \epsilon$.
\end{proof}

\section{Bounds on $s(p)$}
In this section we will take a deeper look at $s(p)$ and finish the proofs of Propositions \ref{cpincrease} and \ref{cpineqsecondpair}. \\

Consider $p_1 = 0.5495021777642$ and choose $n$ to be equal to, say, $100$. Now use a computer and equation (\ref{recgeneral}) to calculate $v_{n,p_1}$. By Lemma \ref{expo} we then know that $s(p_1) < s_n(p) + p_1^{n+1} \left(\frac{n}{1-p_1} + \frac{1}{(1-p_1)^2}\right)$ and actually working out the latter shows in this case $s(p_1) < 0$. On the other hand, with $p_2 = 0.54950217776421 = p_1 + 10^{-14}$, one can calculate that $v_{53,p_2} > 52 + p_2$, so that $s(p_2) > s_{53}(p_2) > 0$. The function $s(p)$ is continuous by Lemma \ref{cont}, which implies by the intermediate value theorem that there exists a $p \in [p_1, p_2]$ such that $s(p)$ is exactly equal to $0$. We will show that this value is unique, which finishes the proof of Proposition \ref{cpincrease}. In order to do this, we will use equation (\ref{recgeneral}) to get a recursive formula for $s_{n}(p)$.

\begin{recursived} \label{recursived}
Let $p$ be such that $\frac{1}{2} \le p \le 1$. We then have the following recursive formula, valid for all $n \in \mathbb{N}$: 

\begin{equation*}
s_{n}(p) = s_{n-1}(p) (1 - p^n) - \min\big(0, s_{n-1}(p)\big) np^{n-1}(1-p) + p^n(1-p) - (1-p)^n \label{reccn}
\end{equation*}
\end{recursived}

\begin{proof}
By definition we have $s_{n}(p) = v_{n,p} - n + 1 - p$, so we can just rewrite equation (\ref{recgeneral}) in terms of $s_{n}(p)$ and $s_{n-1}(p)$.

\begin{align*}
s_{n}(p) &= v_{n,p} - n + 1 - p \\
&= \big(v_{n-1,p} + 1 + p^n(n - 1 - v_{n-1,p}) + np^{n-1}(1-p) \max(0, n - 2 + p - v_{n-1,p}) - (1-p)^n\big) \\
&\hspace{45pt}- n + 1 - p \\
&= v_{n-1,p} - n + 2 - p + p^n\big(1 - p - s_{n-1}(p)\big) - \min\big(0, s_{n-1}(p)\big)np^{n-1}(1-p) - (1-p)^n \\
&= s_{n-1}(p) (1 - p^n) - \min\big(0, s_{n-1}(p)\big) np^{n-1}(1-p) + p^n(1-p) - (1-p)^n \qedhere
\end{align*}
\end{proof}

We will use the above to show that $s(p)$ is an increasing function of $p$ for $p \in [\frac{1}{2}, p_0]$.

\begin{dipvsdiq} \label{dipvsdiq}
Let $p$ and $q$ be such that $\frac{1}{2} \le q < p < \phi$ and define $C_n := \displaystyle \frac{5}{6} \prod_{j=3}^{n}\left(1 - (j+1)\phi^j\right)$ and $\displaystyle C := \lim_{n \rightarrow \infty} C_n$. Then $C > 0$ and for all $n \in \mathbb{N}$ with $s_{n-1}(p) \le 0$, we have $s_n(p) > s_n(q) + C_n(p-q) > s_n(q) + C(p-q) $. 
\end{dipvsdiq}

\begin{proof}
To prove that $C$ is positive, it is sufficient to show $\displaystyle \prod_{j=4}^{\infty}\left(1 - (j+1)\phi^j\right) > 0$, since $\frac{5}{6}(1 - 4\phi^3) > 0$. By taking the logarithm, this is equivalent with proving that the infinite sum $\displaystyle \sum_{j=4}^{\infty} \log\left(1 - (j+1)\phi^j\right)$ converges. And this in turn follows from the inequality $\log(1 - x) > -2x$ which, as mentioned before, holds for all $x < 0.79$. \\

As for comparing $s_n(p)$ to $s_n(q)$, for $n = 2$ we know $s_n(p) - s_n(q) = (2p - p^3) - (2q - q^3) = (p-q)(2 - p^2 - pq - q^2) > \frac{5}{6}(p-q)$, as $2 - p^2 - pq - q^2 > 2 - 3\phi^2 > \frac{5}{6}$. Now assume by induction that $s_{n-1}(p) > s_{n-1}(q) + C_{n-1}(p-q)$, which we just proved to be true for $n = 3$ as the product is then empty. Further assume $s_{n-1}(p) \le 0$ as the Lemma is otherwise vacuous, which by induction implies that $s_{n-1}(q)$ is smaller than $0$ as well. Now we calculate $s_{n}(p)$ by applying Lemma \ref{recursived}.

\begin{align*}
s_{n}(p) &= s_{n-1}(p) \big(1 - p^n - np^{n-1}(1-p)\big) + p^n(1-p) - (1-p)^n \\
&> \big(s_{n-1}(q)  + C_{n-1}(p-q)\big) \big(1 - q^n - nq^{n-1}(1-q)\big) + q^n(1-q) - (1-q)^n \\
&> s_{n-1}(q)\big(1 - q^n - nq^{n-1}(1-q)\big) + q^n(1-q) - (1-p)^q + C_n(p-q) \\
&= s_{n}(q) + C_n(p-q) \qedhere
\end{align*}

\end{proof}

\begin{proof}[Proof of Proposition \ref{cpincrease}]
The convergence of $s_{n}(p)$ to $s(p)$ is already dealt with, as is the existence of the constant $p_0$. The only thing we still need to do is show uniqueness; there is only one $p$ with $\frac{1}{2} \le p < 1$ for which $s(p) = 0$. If $\phi \le p < 1$, then $s(p) > s_{2}(p) \ge 0$, so assume by contradiction $p \in [\frac{1}{2}, \phi)$ is different from $p_0$ while $s(p) = s(p_0) = 0$. Since $s_{n}(p)$ and $s_{n}(p_0)$ converge to $0$ from below, there exists an $n$ such that both $s_{n}(p)$ and $s_{n}(p_0)$ lie in the interval $(-C|p - p_0|, 0)$. But Lemma \ref{dipvsdiq} implies $|s_{n}(p) - s_{n}(p_0)| > C|p - p_0|$, which contradicts the fact that $s_{n}(p)$ and $s_{n}(p_0)$ must both be between $-C|p - p_0|$ and $0$.
\end{proof}

Now we move on to the proof of Proposition \ref{cpineqsecondpair}. Let $n \ge 2$ be arbitrary. Then $s_{n}(p_0) < s(p_0) = 0$, so there exists an $\epsilon > 0$ with $s_{n}(p_0) = -\epsilon$. Since $s_{n}(p)$ is a continuous function of $p$ for this value of $n$, there exists a $\delta > 0$ such that for all $p \in (p_0, p_0 + \delta)$ we get $|s_{n}(p) - s_{n}(p_0)| < \epsilon$. This in turns implies $s_{n}(p) = s_{n}(p) - s_{n}(p_0) + s_{n}(p_0) \le |s_{n}(p) - s_{n}(p_0)| + s_{n}(p_0) < \epsilon - \epsilon = 0$. By the definition of $p_0$ and the fact that $p$ is larger than $p_0$, we know $s(p) > 0$, implying that there exists a smallest integer $n(p) \ge 2$ with $s_{n(p)}(p) \ge 0$. And by construction we have $n(p) > n$. We thusly obtain the following. 

\begin{nparbitrarilylarge} \label{nplarge}
For every $n \in \mathbb{N}$ there exists a $\delta > 0$ such that for all $p \in (p_0, p_0 + \delta)$ we have $n(p) > n$.
\end{nparbitrarilylarge}

\begin{proof}[Proof of Proposition \ref{cpineqsecondpair}]
 Let $\epsilon > 0$ be given and choose $\delta_1 > 0$ small enough such that for all $p \in (p_0, p_0 + \delta_1)$ we get both $s(p) < \left(\frac{1-p}{2}\right)\epsilon$ and $\left(\frac{1 -p}{p}\right)^{n(p)} < \left(\frac{1-p}{2}\right)\epsilon$. The first is possible by Corollary \ref{cont} and the second is possible by Lemma \ref{nplarge}. This then implies for all $n$ that $n - v_{n,p} = 1 - p - s_{n}(p) > 1 - p - s(p) > 1 - p - \left(\frac{1-p}{2}\right)\epsilon$, and for $n \ge n(p)$ that $(1-p)^n < \left(\frac{1-p}{2}\right)\epsilon p^n$. We then get:

\begin{align*}
s(p) &> s(p) - (v_{n(p),p} - n(p) + 1 - p) \\
&= \lim_{N \rightarrow \infty} (v_{N,p} - N + 1 - p) - (v_{n(p),p} - n(p) + 1 - p) \\
&= \lim_{N \rightarrow \infty} \sum_{k = n(p)+1}^{N} v_{k,p} - v_{k-1,p} - 1 \\
&= \sum_{k = n(p)+1}^{\infty} \Big(p^k(k - 1 - v_{k-1,p}) - (1-p)^k\Big) \\
&> \sum_{k = n(p)+1}^{\infty} \Big(p^k\left(1 - p - \left(\frac{1-p}{2}\right)\epsilon\right) - \left(\frac{1-p}{2}\right)\epsilon p^k\\
&= \sum_{k = n(p)+1}^{\infty} (1 - \epsilon)(1 - p)p^k \\
&= (1 - \epsilon)p^{n(p)+1}
\end{align*}

On the other hand, note that from equation (\ref{recgeneral}) it follows that $v_{n,p} - v_{n-1,p} - 1 < (n+1)p^n$. This in turn implies the following inequality:

\begin{align*}
v_{n,p} &= v_{n-1,p} + 1 + p^n(n - 1 - v_{n-1,p}) + np^{n-1}(1-p) \max(0, n - 2 + p - v_{n-1,p}) - (1-p)^n \\
&< v_{n-1,p} + 1 + p^n(n - v_{n,p}) + np^{n-1}(1-p) \max(0, n - 1 + p - v_{n,p}) + (n+1)^2p^{2n}
\end{align*}

Furthermore note that for $n \ge n(p)$, the term $np^{n-1}(1-p) \max(0, n - 1 + p - v_{n,p})$ is equal to $0$. Now choose $\delta_2 > 0$ small enough so that $(n(p)+1)^2p^{n(p)} < \epsilon$ for all $p \in (p_0, p_0 + \delta_2)$. We can then upper bound $s(p)$ as follows:

\begin{align*}
s(p) &< s(p) - (v_{n(p)-1,p} - (n(p) - 1) + 1 - p) \\
&< p^{n(p)}(n(p) - v_{n(p),p}) + (n(p)+1)^2p^{2n(p)} + s(p) - (v_{n(p),p} - n(p) + 1 - p)  \\
&= p^{n(p)}(n(p) - v_{n(p),p}) + (n(p)+1)^2p^{2n(p)} + \sum_{k = n(p)+1}^{\infty} \Big(p^k(k - 1 - v_{k-1,p}) - (1-p)^k\Big) \\
&< p^{n(p)}(1-p) + \epsilon p^{n(p)} + \sum_{k = n(p)+1}^{\infty} (1-p)p^k \\
&= (1 + \epsilon)p^{n(p)}
\end{align*}

Taking $\delta = \min(\delta_1, \delta_2)$ finishes the proof.
\end{proof}

\newpage
\section{Differentiability}
This section is devoted to proving Proposition \ref{differentiable}; differentiability of $s(p)$ in the point $p = p_0$. In order to do this, we first have to prove that the quotient $\frac{s_n(p) - s_n(q)}{p - q}$ is bounded. Lemma \ref{dipvsdiq} already provides a conditional lower bound on this quotient and it furthermore shows that if $s'(p_0)$ exists, then it must be positive. We shall also need a corresponding upper bound.

\begin{dipvsdiqtwo} \label{dipvsdiqtwo}
Let $p$ and $q$ be such that $\frac{1}{2} \le q < p < \phi$. We then have the inequality $s_n(p) < s_n(q) + 17(p-q)$ for all $n \in \mathbb{N}$.
\end{dipvsdiqtwo}
\begin{proof}
Let us use the definition of $s_n(p)$ and start calculating.

\begin{align*}
s_n(p) - s_n(q) &= q - p + v_{n,p} - v_{n,q} \\
&= q - p + \sum_{k=1}^n (v_{k,p} - v_{k-1,p} - 1) - \sum_{k=1}^n (v_{k,q} - v_{k,q-1} - 1) \\
&= q - p + \sum_{k=1}^n p^k(k - 1 - v_{k-1,p}) + kp^{k-1}(1-p)\max(0, k-2+p-v_{k-1,p}) - (1-p)^k \\
&- \sum_{k=1}^n q^k(k - 1 - v_{k-1,q}) + kq^{k-1}(1-q)\max(0, k-2+q-v_{k-1,q}) - (1-q)^k \\
&< q - p + \sum_{k=1}^n p^k(k - 1 - v_{k-1,p}) + kp^{k-1}(1-p)\max(0, k-2+p-v_{k-1,p}) - (1-p)^k \\
&- \sum_{k=1}^n q^k(k - 1 - v_{k-1,p}) + kq^{k-1}(1-q)\max(0, k-2+p-v_{k-1,p}) - (1-q)^k \\
&< q - p + \sum_{k=1}^n p^k - q^k + \sum_{k=1}^n k\big(p^{k-1}(1-p) - q^{k-1}(1-q)\big) + \sum_{k=1}^n (1-q)^k - (1-p)^k \\
&< q - p + \sum_{k=1}^{\infty} p^k - q^k + \sum_{k=1}^{\infty} k\big(p^{k-1}(1-p) - q^{k-1}(1-q)\big) + \sum_{k=1}^{\infty} (1-q)^k - (1-p)^k \\
&= q - p + \frac{1}{1-p} - \frac{1}{1-q} + \frac{1}{1-p} - \frac{1}{1-q} + \frac{1}{q} - \frac{1}{p} \\
&= q - p + \frac{2(p-q)}{(1-p)(1-q)} + \frac{p-q}{pq} \\
&< 17(p-q) \qedhere
\end{align*}
\end{proof}

It is interesting to note that Lemmas \ref{dipvsdiq} and \ref{dipvsdiqtwo} together are already sufficient to prove the estimate $n(p) = \frac{\log(p - p_0)}{\log(p_0)} + O(1)$. Moreover, in what follows we will actually need a lower bound of this form. So let us quickly state and prove it. 

\begin{nplower} \label{nplower}
Let $0 < \delta < \frac{1}{21}$ be small enough so that $\frac{1}{2}p^{n(p)} < s(p)$ for all $p \in (p_0, p_0 + \delta) \subseteq (p_0, 0.6)$. Then $n(p) > \frac{\log(p - p_0)}{\log(p_0)} - 7$ for all these values of $p$.
\end{nplower}

\begin{proof}
Applying Lemma \ref{dipvsdiqtwo} we get $\frac{1}{2}p^{n(p)} < s(p) < 17(p - p_0)$. Doubling both sides and taking logarithms we get $n(p) > \frac{\log(p - p_0)}{\log(p)} + \frac{\log(34)}{\log(p)} > \frac{\log(p - p_0)}{\log(p_0)} - 7$.
\end{proof}

In particular we see that $p_0^{n(p)+1} < p_0^{-6}(p-p_0) < 37(p - p_0)$, which we shall use near the end of the upcoming proof. \\
 
Now, it is important to realize that $s_n(p_0)$ is by equation (\ref{recgeneral}) a polynomial $P(p_0)$ in $p_0$, and if $n(p) > n$ for all $p \in (p_0, p_0 + \delta)$, then $s_n(p) = P(p)$ is the same polynomial for all $p \in [\frac{1}{2}, p_0 + \delta)$. This implies that the derivative $s'_{n}(p_0) = P'(p_0)$ exists and can therefore be arbitrarily well approximated by $\frac{s_{n}(p) - s_{n}(p_0)}{p - p_0}$. \\

\begin{proof}[Proof of Proposition \ref{differentiable}]
Let $\epsilon > 0$ be given, choose $n$ large enough so that \mbox{$10^4n^2 \cdot 0.6^n$} $< \frac{1}{2}\epsilon$ and let $p$ for the rest of this proof be different from $p_0$. Furthermore, choose $\delta < \frac{1}{21}$ positive and small enough such that for all $p \in (p_0, p_0 + \delta)$ we have $n(p) > n$ and $\frac{1}{2}p^{n(p)} < s(p)$, and for all $p \in (p_0 - \delta, p_0 + \delta) \subseteq (0.5, 0.6)$ we have $\left|\frac{s_n(p) - s_n(p_0)}{p - p_0} - s'_n(p_0)\right| < \frac{1}{2}\epsilon$. Then we claim $\left|\frac{s(p) - s(p_0)}{p - p_0} - s'_n(p_0)\right| < \epsilon$ for all $p \in (p_0 - \delta, p_0 + \delta)$. This would prove that both $\displaystyle \lim_{n \rightarrow \infty} s'_n(p_0)$ and $\displaystyle \lim_{p \rightarrow p_0} \frac{s(p) - s(p_0)}{p - p_0}$ exist, and that they are equal. And this limit may then be appropriately denoted by $s'(p_0)$. Let us first apply the triangle inequality. \\

\begin{align*}
\left|\frac{s(p) - s(p_0)}{p - {p_0}} - s'_{n}(p_0)\right| &\le  \left|\frac{s(p) - s(p_0)}{p - {p_0}} - \frac{s_{n}(p) - s_{n}(p_0)}{p - p_0} \right| + \left|\frac{s_{n}(p) - s_{n}(p_0)}{p - {p_0}} - s'_{n}(p_0)\right| \\
&<  \left|\frac{s(p) - s(p_0)}{p - {p_0}} - \frac{s_{n}(p) - s_{n}(p_0)}{p - {p_0}} \right| + \frac{1}{2}\epsilon \\
&= \left|\frac{\big(s(p) - s_{n}(p)\big) - \big(s(p_0) - s_n(p_0)\big)}{p - {p_0}} \right| + \frac{1}{2}\epsilon
\end{align*}

Since $10^4n^2 \cdot 0.6^n < \frac{1}{2}\epsilon$, it is sufficient to show $\left|\big(s(p) - s_n(p)\big) - \big(s(p_0) - s_{n}(p_0)\big)\right| < 10^4n^2 \cdot 0.6^n \cdot |p - p_0|$. And this is exactly what we will do. \\

We will write $s(p) - s_n(p)$ as a combination of three sums $-\Sigma_{1,p} - \Sigma_{2,p} + \Sigma_{3,p}$ so that $\left|\big(s(p) - s_n(p)\big) - \big(s(p_0) - s_{n}(p_0)\big)\right| \le \displaystyle \sum_{i=1}^3|\Sigma_{i,p} - \Sigma_{i,p_0}|$, and we will handle the three terms $|\Sigma_{i,p} - \Sigma_{i,p_0}|$ separately. To get those three sums, write $s(p) - s_n(p)$ as an infinite telescoping series and apply Lemma \ref{recursived}.

\begin{align*}
s(p) - s_n(p) &= \sum_{k = n+1}^{\infty} s_{k}(p) - s_{k-1}(p) \\
&= -\sum_{k = n+1}^{\infty} s_{k-1}(p)p^k - \sum_{k = n+1}^{\infty} \min\big(0, s_{k-1}(p)\big)kp^{k-1}(1-p) + \sum_{k = n+1}^{\infty} p^k(1-p) - (1-p)^k \\
&:= -\Sigma_{1,p} - \Sigma_{2,p} + \Sigma_{3,p}
\end{align*}

To analyze the three terms $|\Sigma_{i,p} - \Sigma_{i,p_0}|$ we will come across differences like $p^n - p_0^n$ a lot. We will bound such differences by factorizing them and applying $\max(p, p_0) < 0.6$. 

\begin{align*}
|p^n - p_0^n| &= |(p - p_0)(p^{n-1} + p_0p^{n-2} + \ldots + p_0^{n-2}p + p_0^{n-1})| \\
&< |(p - p_0) (n \cdot \max(p,p_0)^{n-1})| \\
&< 2n \cdot 0.6^n \cdot |p - p_0|
\end{align*}

With this in mind, let us look at $\left|\Sigma_{3,p} - \Sigma_{3,p_0}\right|$ first, as it provides a good warm-up for the rest.

\begin{align*}
\left|\Sigma_{3,p} - \Sigma_{3,p_0}\right| &=\left|\sum_{k = n+1}^{\infty} p^k(1-p) - (1-p)^k - \sum_{k = n+1}^{\infty} p_0^k(1-p_0) - (1-p_0)^k\right| \\
&= \left|p^{n+1} - \frac{(1-p)^{n+1}}{p} - p_0^{n+1} + \frac{(1-p_0)^{n+1}}{p_0} \right| \\
&\le \left|p^{n+1} - p_0^{n+1}\right| + \left|\frac{(p - p_0)(1-p_0)^{n+1}}{pp_0} \right| + \left|\frac{(1-p_0)^{n+1} - (1-p)^{n+1}}{p} \right| \\
&< 2n \cdot 0.6^n \cdot |p - p_0| + 4 \cdot 0.6^n \cdot |p - p_0|  + 4n \cdot 0.6^n \cdot |p - p_0| \\
&\le 10n \cdot 0.6^n \cdot |p - p_0|
\end{align*}

The term $\left|\Sigma_{1,p} - \Sigma_{1,p_0}\right|$ takes only slightly longer to analyze, and for this we will use the easy to check bound $|s_{k-1}(p_0)| < 1$ and the upper bound $|s_{k-1}(p) - s_{k-1}(p_0)| < 17|p - p_0|$ from Lemma \ref{dipvsdiqtwo}.

\begin{align*}
\left|\Sigma_{1,p} - \Sigma_{1,p_0} \right| &= \left|\sum_{k = n+1}^{\infty} s_{k-1}(p)p^k - \sum_{k = n+1}^{\infty} s_{k-1}(p_0)p_0^k\right| \\
&= \left|\sum_{k = n+1}^{\infty} \big(s_{k-1}(p) - s_{k-1}(p_0)\big)p^k + \sum_{k = n+1}^{\infty} s_{k-1}(p_0)(p^k - p_0^k) \right| \\
&< \left|\sum_{k = n+1}^{\infty} 17p^k(p-p_0) \right| + \left| \sum_{k = n+1}^{\infty} (p^k - p_0^k) \right|\\
&= \left|\frac{17(p-p_0)p^{n+1}}{1-p}\right| + \left|\frac{p^{n+1}}{1-p} - \frac{p_0^{n+1}}{1-p_0} \right|\\
&\le \left|\frac{17(p-p_0)p^{n+1}}{1-p}\right| + \left|\frac{p^{n+1} - p_0^{n+1}}{(1-p)(1-p_0)}\right| + \left|\frac{pp_0(p_0^{n} - p^{n})}{(1-p)(1-p_0)}\right| \\
&< 43\cdot 0.6^n \cdot |p - p_0| + 14n \cdot 0.6^n \cdot |p - p_0| + 14n \cdot 0.6^n \cdot |p - p_0| \\
&\le 71n \cdot 0.6^n \cdot |p - p_0|
\end{align*}

Finally, we are left with analyzing $\left|\Sigma_{2,p} - \Sigma_{2,p_0}\right|$. This is the trickiest one as it contains $\min\big(0, s_{k-1}(p)\big)$ which is either equal to $s_{k-1}(p)$ for all $k$ if $p < p_0$, but will, if $p > p_0$, be equal to $0$ for all $k > n(p)$. We therefore need to handle these two different cases separately. But let us first split up $\left|\Sigma_{2,p} - \Sigma_{2,p_0}\right|$ into two parts.

\begin{align*}
\left|\Sigma_{2,p} - \Sigma_{2,p_0}\right| &= \left|\sum_{k = n+1}^{\infty} \min\big(0, s_{k-1}(p)\big)kp^{k-1}(1-p) - \sum_{k = n+1}^{\infty} s_{k-1}(p_0)kp_0^{k-1}(1 - p_0) \right|\\
&\le \left|\sum_{k = n+1}^{\infty} \Big(\min\big(0, s_{k-1}(p)\big) - s_{k-1}(p_0)\Big)kp^{k-1}(1-p)\right| \\
&+ \left|\sum_{k = n+1}^{\infty} s_{k-1}(p_0)k\big(p^{k-1}(1-p) - p_0^{k-1}(1-p_0)\big) \right| 
\end{align*}

As for the second sum here, we can use the bound $|s_{k-1}(p_0)| < 1$ and deal with it regardless of whether $p$ is smaller or larger than $p_0$.

\begin{align*}
\left|\sum_{k = n+1}^{\infty} k\big(p^{k-1}(1-p) - p_0^{k-1}(1-p_0)\big) \right| &= \left|(n+1)p^n + \frac{p^{n+1}}{1-p} - (n+1)p_0^n - \frac{p_0^{n+1}}{1-p_0}\right| \\
&\le \left|(n+1)(p^n - p_0^n)\right| + \left|\frac{p^{n+1} - p_0^{n+1}}{(1-p)(1-p_0)} \right| + \left|\frac{pp_0(p_0^n - p^n)}{(1-p)(1-p_0)} \right| \\
&< 4n^2 \cdot 0.6^n \cdot \left|p - p_0\right| + 14n \cdot 0.6^n \cdot \left|p - p_0\right| + 14n \cdot 0.6^n \cdot \left|p - p_0\right| \\
&\le 32n^2 \cdot 0.6^n \cdot \left|p - p_0\right|
\end{align*}

The first part of $\left|\Sigma_{2,p} - \Sigma_{2,p_0}\right|$ does depend on the value of $\min\big(0, s_{k-1}(p)\big)$, so first assume $p < p_0$. In that case we can	 apply Lemma \ref{dipvsdiqtwo} again.

\begin{align*}
\left|\sum_{k = n+1}^{\infty} \Big(\min\big(0, s_{k-1}(p)\big) - s_{k-1}(p_0)\Big)kp^{k-1}(1-p) \right| &= \left|\sum_{k = n+1}^{\infty} \big(s_{k-1}(p) - s_{k-1}(p_0)\big)kp^{k-1}(1-p) \right| \\
&\le \sum_{k = n+1}^{\infty} 17|p - p_0|kp^{k-1}(1-p) \\
&= 17\left((n+1)p^n + \frac{p^{n+1}}{1-p}\right)|p - p_0| \\
&\le 85n \cdot 0.6^n \cdot \left|p - p_0\right| 
\end{align*}

If, on the other hand, $p_0 > p$, then $s_{k-1}(p) \ge 0$ for $k > n(p)$. For such $k$ we then need to apply $|s_{k-1}(p_0)| < 8kp_0^k$ (which follows from Lemma \ref{expo}) and $p_0^k \le p_0^{n(p)+1} \le 37(p - p_0)$ which, as mentioned before we started this whole proof, follows from Corollary \ref{nplower}. We furthermore need one lemma that one can prove by differentiating the equality from Lemma \ref{expo}.

\begin{expok}
For all $p \in (0.5, 0.6)$ and all $n \in \mathbb{N}$ we have the following equality and subsequent inequality:
\begin{align*}
\sum_{k = n+1}^{\infty} k^2p^{k-1}(1-p) &= n(n+1)p^n + \frac{np^{n+1}}{1-p} + \frac{(n+1)p^n}{1-p} + \frac{2p^{n+1}}{(1-p)^2} \\
&< 24n^2 \cdot 0.6^n
\end{align*}
\end{expok}

Let us use this inequality to bound the remaining part of $\left|\Sigma_{2,p} - \Sigma_{2,p_0}\right|$.

\begin{align*}
&\hspace{12pt} \left|\sum_{k = n+1}^{\infty} \Big(\min\big(0, s_{k-1}(p)\big) - s_{k-1}(p_0)\Big)kp^{k-1}(1-p) \right| \\
&\le \left|\sum_{k = n+1}^{n(p)} \big(s_{k-1}(p) - s_{k-1}(p_0)\big)kp^{k-1}(1-p)\right| + \left|\sum_{k = n(p)+1}^{\infty} s_{k-1}(p_0)kp^{k-1}(1-p)\right| \\
&< \left|\sum_{k = n+1}^{n(p)} 17kp^{k-1}(1-p)(p-p_0)\right| + \left|\sum_{k = n(p)+1}^{\infty} (8kp_0^k)kp^{k-1}(1-p)\right| \\
&< \left|\sum_{k = n+1}^{n(p)} 17kp^{k-1}(1-p)(p-p_0)\right| + \left|\sum_{k = n(p)+1}^{\infty} 300 k^2p^{k-1}(1-p)(p-p_0)\right| \\
&< \left|\sum_{k = n+1}^{\infty} 300k^2p^{k-1}(1-p)(p - p_0) \right| \\
&< 7200n^2 \cdot 0.6^n \cdot \left|p - p_0\right|
\end{align*}

 By combining all estimates, we can finally finish the calculation from the start of this proof.

\begin{align*}
\left|\frac{s(p) - s(p_0)}{p - p_0} - s'_{n}(p_0)\right| &< \left|\frac{\big(s(p) - s_n(p)\big) - \big(s(p_0) - s_n(p_0)\big)}{p - p_0} \right| + \frac{1}{2}\epsilon \\
&\le \sum_{i=1}^3 \left| \frac{\Sigma_{i,p} - \Sigma_{i,p_0}}{p - p_0}\right| + \frac{1}{2}\epsilon \\
&< (71n + 32n^2 + \max(7200n^2, 85n) + 10n) \cdot 0.6^n + \frac{1}{2}\epsilon \\
&< 10^4n^2 \cdot 0.6^n +  \frac{1}{2}\epsilon \\
&< \epsilon \qedhere
\end{align*}
\end{proof}

\section{When the probability of landing heads is smaller than $\frac{1}{2}$}
When $0 < p < \frac{1}{2}$ we are unfortunately not able to prove anything resembling Theorem \ref{Main}. This is due to the fact that there are infinitely many $n$ for which the inequality $v_{n,p} \ge v_{n-1,p} + 1$ from Proposition \ref{vnincrease} does not hold.

\begin{vnnotincrease} \label{notincrease}
For every $p \in (0, \frac{1}{2})$, there are infinitely many $n$ such that $v_{n,p} < v_{n-1,p} + 1$.
\end{vnnotincrease}

\begin{proof}
Let $n_0$ be a given arbitrary positive integer. We will prove that there exists an $n_1 > n_0$ such that $v_{n_1,p} < v_{n_1-1,p} + 1$. Let $n$ be large enough such that the probability of getting at least $n - n_0$ heads, denoted by $P$, is less than $\frac{(1-p)^n}{n}$. Such an $n$ exists precisely when $p < \frac{1}{2}$. We may further assume that the inequality $v_{k,p} \ge v_{k-1,p} + 1$ is satisfied for all $k$ with $n_0 < k \le n$, or otherwise we are done immediately. These inequalities imply that for $j < n - n_0$ we get $\max_{1 \le i \le j} (v_{n-i,p} + i) = v_{n-1,p} + 1$. Now we look at equation (\ref{recgeneral}) and use the fact that $v_{n-i,p} + i \le n$ for all $i \ge n - n_0$. For ease of notation, define $\Sigma_1$ to be the part of the sum where $j$ goes from $1$ to $n - n_0 - 1$ and let $\Sigma_2$ be the part of the sum where $j$ goes from $n - n_0$ to $n-1$.

\begin{align*}
v_{n,p} &= np^n + v_{n-1,p} (1-p)^n + \sum_{j=1}^{n-1} \binom{n}{j}p^j(1-p)^{n-j} \left(\max_{1 \le i \le j} (v_{n-i,p} + i)\right) \\
&= np^n + v_{n-1,p} (1-p)^n + \Sigma_1 + \Sigma_2 \\
&\le v_{n-1,p} (1-p)^n + (1 - P - (1-p)^n)(v_{n-1,p} + 1) + nP \\
&< v_{n-1,p} (1-p)^n + (1 - (1-p)^n)(v_{n-1,p} + 1) + (1-p)^n \\
&= v_{n-1,p} + 1 \qedhere
\end{align*}

\end{proof}

All is not lost however. For example, we still get that $s_n(p)$ converges for all $p > 0$.

\begin{cpsmallerthanhalf}
For every $p \in (0, 1)$ there exists a $s(p) \in [1 - p - \frac{1}{p}, 1-p)$ such that $s_n(p)$ converges to $s(p)$.
\end{cpsmallerthanhalf}

\begin{proof}
By equation (\ref{recform}) we always have $v_{n,p} \ge v_{n-1,p} + 1 - (1-p)^n$. This implies that we get the following bound:

\begin{align*}
v_{n,p} &= \sum_{i=1}^n v_{i,p} - v_{i-1,p} \\ 
&\ge n - \sum_{i=1}^n (1-p)^i \\
&> n - \sum_{i=1}^{\infty} (1-p)^i \\
&= n - \frac{1}{p}
\end{align*}

We conclude that for all $n \in \mathbb{N}$, $s_n(p)$ is bounded below by $1 - p - \frac{1}{p}$ and bounded above by $1-p$. This implies that $\displaystyle \limsup_{n \rightarrow \infty} s_n(p)$ exists, and we will call this limit superior $s(p)$. We will show that this limit superior is actually a limit. For this it is sufficient that for every $\epsilon > 0$ there exists an $N$ such that for all $n \ge N$ we get $s_n(p) > s(p) - \epsilon$. \\

Let $\epsilon$ be given and choose $n_1$ large enough such that $\frac{(1-p)^{n_1}}{p} < \frac{\epsilon}{2}$. Let $n_2 \ge n_1$ be an integer which is close to the limit superior in the sense that $s_{n_2}(p) > s(p) - \frac{\epsilon}{2}$. Then for all $n \ge n_2$ we get:

\begin{align*}
s_n(p) &= s_{n_2}(p) + \sum_{i=n_2 + 1}^n (v_{i+1,p} - v_{i,p} - 1) \\
&\ge s_{n_2}(p) - \sum_{i=n_2 + 1}^n (1-p)^i \\
&> s_{n_2}(p) - \sum_{i=n_2 + 1}^{\infty} (1-p)^i \\
&= s_{n_2}(p) - \frac{(1-p)^{n_2+1}}{p} \\
&> s(p) - \epsilon \qedhere
\end{align*}
\end{proof}

\newpage
\section{Concluding thoughts and remarks}
Peter Pfaffelhuber remarked in personal communication that equation (\ref{recformtwo}) implies by induction the following non-recursive formula for $v_{n,p}$, which holds for $p \ge \phi$:

\begin{equation*}
v_{n,p} = n - \sum_{k=1}^n \left[(1-p)^k \prod_{j=k+1}^n (1-p^j) \right]
\end{equation*}

\begin{proof}
This equation can be checked for $n = 0$, so let us assume it is true for $m = n-1$. We will then prove it to be true for $m = n$ as well.

\begin{align*}
v_{n,p} &= v_{n-1,p} + 1 + p^n(n - 1 - v_{n-1,p}) - (1-p)^n \\
&= n - (1 - p^n)\sum_{k=1}^{n-1} \left[(1-p)^k \prod_{j=k+1}^{n-1} (1-p^j)\right] - (1-p)^n \\
&= n - \sum_{k=1}^{n-1} \left[(1-p)^k \prod_{j=k+1}^{n} (1-p^j)\right] - (1-p)^n \\
&= n - \sum_{k=1}^n \left[(1-p)^k \prod_{j=k+1}^n (1-p^j) \right]
\end{align*}

Here, the last equality follows since the product is empty for $k = n$.
\end{proof}

Peter furthermore realized that, if the goal of the game instead is to optimize the probability of ending up with all heads (as opposed to maximizing the expected number of heads), it is for $p > \frac{1}{2}$ always optimal to set aside only one coin, unless all coins show heads. With $w_{n,p}$ the probability of getting all heads with optimal play, this follows from the equalities $w_{n-1,p} < w_{n,p} < \dfrac{p^{n+1}}{(1-p)^{n+1} + p^{n+1}}$. These inequalities can be proven via induction in much the same way as the proof of Lemma \ref{boundsforphi}, by applying the following recursive formula:

\begin{equation*}
w_{n,p} = p^n + \sum_{j=1}^{n-1} \binom{n}{j}p^j(1-p)^{n-j} \left(\max_{1 \le i \le j} w_{n-i,p} \right)
\end{equation*}

Returning to the original game, it is possible to state Proposition \ref{differentiable} a bit more generally, as for most $p \in (\frac{1}{2}, 1]$ we can prove differentiability of $s(p)$ similarly to how we proved it for $p_0$. It is interesting to note however, that it does not work for all $p > \frac{1}{2}$. Recall that we used the fact that $s_n(p_0)$ can be written as a polynomial $P(p_0)$ and that there exists an interval around $p_0$ such that for all $q$ in this interval it holds true that $s_n(q) = P(q)$. However, there do exist values of $p > p_0$ for which this condition is false. Those $p$ are precisely the values for which an $n$ exists where $s_n(p)$ is exactly equal to $0$, implying $v_{n,p} + 1 = v_{1,p} + n$. For such $p$ we have $n(q) > n(p)$ for all $q \in (p_0, p)$, so that $s_n(q)$ and $s_n(p)$ are not the same polynomial, and the limit $\displaystyle \lim_{q \rightarrow p} \dfrac{s_n(p) - s_n(q)}{p - q}$ does not have to exist. It might not come as a surprise that $\phi$ is the largest element in this countable set of exceptional values of $p \in (\frac{1}{2}, 1]$ where $s(p)$ is possibly not differentiable. \\ 

Moving on to smaller values of $p$, it follows from Lebesgue's theorem on the differentiability of monotone functions that $s(p) + p$ \big(and therefore $s(p)$\big) is differentiable almost everywhere, in particular on $(0, \frac{1}{2})$ as well. It is therefore tempting to conjecture more generally that the exceptional set of $p$ where $s(p)$ is not differentiable consists of all $p \in (0, 1]$ for which distinct $n$ and $m$ exist with $v_{n,p} + m = v_{m,p} + n$. We have not really thought about this much, however. \\

Possibly more importantly, for most $p < \frac{1}{2}$ it seems likely (based on computer calculations) that there are infinitely many $n$ for which $v_{n,p} > v_{n-1,p} + 1$ is true. It would therefore be very nice if this could be proven (or refuted). And if this were indeed true, then it does suggest that the optimal strategy for such $p$ is, in a sense, infinitely complex. To elaborate on this, one of the first non-trivial questions one can ask is: how many coins should you set aside if you get exactly two heads? But if there are infinitely many $n$ such that $v_{n,p} > v_{n-1,p} + 1$ and we know by Lemma \ref{notincrease} that there infinitely many $n$ with $v_{n,p} < v_{n-1,p} + 1$ as well, then even the answer to this seemingly simple question changes from 'one' to 'two' and back infinitely often.

\section{Acknowledgements}
The author is thankful for interesting and inspiring comments by Joachim Breitner, Peter Pfaffelhuber and Anneroos Everts. In particular for Joachim to ask the question in the first place, starting a formal verification of the results in this paper, and providing the first proof of Lemma \ref{vnmonotone}. In particular for Peter to notice the necessity of Lemma \ref{vnmonotone}, finding a non-recursive formula for $v_{n,p}$ in the case $p \ge \phi$, and considering the alternative objective of maximizing the probability of ending up with all heads. And in particular for Anneroos for being the best reviewer one can hope for.

\end{document}